\documentclass[12pt, reqno]{amsart}
\usepackage{amsmath}  \hoffset=-56pt

\usepackage{amsmath, amsfonts, rawfonts,amssymb,txfonts, enumerate}
\usepackage{eucal}
\usepackage{latexsym}
\usepackage{cite, color}
 \usepackage{bm}
\usepackage{appendix}

\newcommand{\rnnn}{\mathbb R^{n}}
\newcommand{\m}{\mathbb R^{n-1}}

\newcommand{\sn}{ {\mathbb{S}^{n-1}}}

\newcommand{\psum}{{+_{\negthinspace\kern-2pt p}}\,}
\newcommand{\qsum}[1]{{+_{\negthinspace\kern-2pt #1}}\,}
\newcommand{\dpsum}{{\tilde+_{\negthinspace\kern-1pt p}}\,}
\newcommand{\dqsum}[1]{{\tilde+_{\negthinspace\kern-1pt #1}}\,}
\newcommand{\lsub}[1]{\hskip -1.5pt\lower.5ex\hbox{$_{#1}$}}

\numberwithin{equation}{section}

\newtheorem{theo}{Theorem}[section]

\newtheorem{lem}[theo]{Lemma}
\newtheorem{prop}[theo]{Proposition}

 \theoremstyle{definition}

\begin{document}

\title{Some remarks on a class of logarithmic curvature flow}

\author[J. Hu]{Jinrong Hu}
\address{School of Mathematics, Hunan University, Changsha, 410082, Hunan Province, China}
\email{hujinrong@hnu.edu.cn}

\author[Q. Mao]{Qiongfang Mao}
\address{School of Mathematics, Hunan University, Changsha, 410082, Hunan Province, China}
\email{maoqiongfang@hnu.edu.cn}


\begin{abstract}
In this paper, we introduce a class of new logarithmic curvature flow. The flows are designed to embrace the monotonicity of the related functional, and the convergence of this flow would tackle the solvability of the weighted Christoffel-Minkowski problem, but a full proof scheme is missing, the key factor of forming this phenomenon lies in the establishment of the upper bound of the principal curvature, which essentially depends on finding a clean condition on smooth positive function defined on the unit sphere $\sn$. Except for obtaining this tricky estimate, we get all the other a priori estimates and hope that this note can attract wide attention to this interesting issue.
\end{abstract}

\maketitle

\baselineskip18pt

\parskip3pt

\section{Introduction}
\label{Sec1}
The Christoffel-Minkowski problem is the problem of finding a convex hypersurface with its $k$-th area measure prescribed on $\sn$. If the boundary of the convex body is smooth, this problem leads to finding convex solutions of the nonlinear Hessian equation
\begin{equation}\label{CR}
\sigma_{k}(\nabla^{2}_{\sn}h+hI)(x)=\frac{1}{f(x)},\quad {\rm on} \ \sn,
\end{equation}
where $h:\sn \rightarrow (0,\infty)$ is the unknown function, $\nabla^{2}_{\sn}h$ is the spherical Hessian of $h$ with respect to a local orthonormal frame on $\sn$, $I$ is the identity matrix, $\sigma_{k}$ is the $k$-th elementary symmetric function for principal curvature radii and $f$ is a positive smooth function. For simplicity, in the subsequence, we abbreviate $\sigma_{k}(\nabla^{2}_{\sn}h+hI)(x)$ as $\sigma_{k}(x)$.

In view of \eqref{CR}, if $k=1$, it reduces to the Christoffel problem. The early proof schemes in this circumstance were  obtained by Christoffel\cite{Cr85}, Hurwitz\cite{Hu02}, Hilbert\cite{Hil53}, S\"{u}ss\cite{S33}. The other special case of \eqref{CR} is $k=n-1$, which is called as the Minkowski problem. This situation has been dealt with by the search of Minkowski\cite{M897}, Aleksandrov \cite{A42}, Nirenberg\cite{N53}, Pogorelov\cite{P78} and Cheng-Yau\cite{CY76}. For $1<k<n-1$, the treatment of intermediate cases become very tricky. Corresponding to $1<k<n-1$, the solvability of \eqref{CR} was first essentially attacked by Guan-Ma\cite[Theorem 1.3]{GM03} with the assistance of a delicate constant rank theorem under the following assumptions:
 \begin{itemize}
  \item [(i)]
\[
\int_{\sn}\frac{h}{f(x)}dx=0,\]

\item [(ii)] \[
\nabla^{2}_{\sn}f^{\frac{1}{k}}+f^\frac{1}{k}I\geq 0,\]

\item [(iii)] There is a continuous (with respect to the $C^{\infty}$-norm) homotopy from $\frac{1}{f}$ to the one constant function.

    \end{itemize}
Notice that condition (i) is also necessary own to the translation invariance of this problem, condition (ii) is sufficient but not necessary, and condition (iii) was removed by a direct verification showed  in Sheng-Trudinger-Wang\cite{STW04}.

Different from the elliptic technique adopted by \cite{GM03} and \cite{STW04}, Bryan-Ivaki-Scheuer \cite{BIS23} recently reprove  the regular  Christoffel-Minkowski problem via a curvature flow involving a global term in the premise that the above assumptions (i) and (ii) hold. In this paper, inspired by Chou-Wang \cite{CW00}, we introduce a class of new logarithmic curvature flow,  employing a family of convex hypersurfaces $\partial \Omega_{t}$ parameterized by smooth map $X(\cdot,t):\sn\rightarrow \rnnn$ satisfying the following flow equation,
\begin{equation}\label{GN}
\frac{\partial X(x,t)}{\partial t}={\rm log}\left( \sigma_{k}(x,t)f(v) \right)v.
\end{equation}
To our knowledge, the utilization of logarithmic curvature flows for tackling the solvability of Minkowski problems is little except in case of the classical Minkowski problem ($k=n-1$) \cite{CW00}. The key ingredient of forming this phenomenon may lie in the establishment of a prior estimates. Our aim is to fill this gap and enrich this topic to some extent.

Using the flow \eqref{GN}, we are devoted to dealing with the solvability of the following regular weighted Christoffel-Minkowski problem,
\begin{equation}\label{WC}
\sigma_{k}(x)=\frac{1}{f(x)e^{\xi\cdot x}}, \ {\rm on} \ \sn,
\end{equation}
where $\xi$ is a point of $\rnnn$.

The existence of the solution to \eqref{WC} relies on a priori estimates of solution to \eqref{GN}. The main obstacle is to establish the upper bound of the principal curvature, which essentially depends on finding a clean condition on $f$. Except for obtaining this tricky estimate, we get all the other a priori estimates, which include the $C^{0}$, $C^{1}$ estimates, the lower bound of the principal curvature, the lower and upper bounds of $\sigma_{k}$.

\section{Preliminaries}

\label{Sec2}
In this section, we collect some basic knowledge on $k$-th Hessian operators.

The $k$-th elementary symmetric function for $\lambda=(\lambda_{1},\cdots,\lambda_{n-1})\in \m$ is defined as
\[
\sigma_{k}(\lambda)=\sum_{1\leq i_{1}<i_{2}<\cdots<i_{k}\leq n-1}\lambda_{i_{1}}\lambda_{i_{2}}\cdots\lambda_{i_{k}},\quad 1\leq k \leq n-1.
\]
Let $\Theta_{n-1}$ be the set of all symmetric $(n-1)\times (n-1)$ metrics. The $k$-th elementary symmetric function for $A\in \Theta_{n-1}$ is
\[
\sigma_{k}(A)=:\sigma_{k}(\lambda(A)),\quad \lambda(A) \ {\rm is \ the \ eigenvalue \ of} \ A.
\]
We say $A\in \Theta_{n-1}$ belongs to $\Gamma_{k}$ if its eigenvalue $\lambda(A)\in \Gamma_{k}$.
The Garding cone $\Gamma_{k}$ is defined as
\[
\Gamma_{k}=\{\lambda\in \m|\sigma_{i}(\lambda)> 0, \ {\rm for } \ 1\leq i\leq k\}.
\]
Here we denote by $\sigma_{k}(\lambda|i)$ the symmetric function with $\lambda_{i}=0$. Now, we list some standard formulas and properities of elementary symmetric functions that we shall use in what follows.
\begin{prop}
Let $\lambda=(\lambda_{1},\cdots, \lambda_{n-1})\in \m$ and $k=0,1,\cdots, n-1$. Then,
\begin{enumerate}
\item[(i)] $\sigma_{k+1}(\lambda)=\sigma_{k+1}(\lambda|i)+\lambda_{i}\sigma_{k}(\lambda|i), \quad \forall 1\leq i\leq n-1$.
\item[(ii)] $\sum_{i=1}^{n-1}\lambda_{i}\sigma_{k}(\lambda|i)=(k+1)\sigma_{k+1}(\lambda)$.

\item[(iii)]$\sum_{i=1}^{n-1}\sigma_{k}(\lambda|i)=(n-k-1)\sigma_{k}(\lambda)$.

\item[(iv)] $\frac{\partial \sigma_{k+1}(\lambda)}{\partial \lambda_{i}}=\sigma_{k}(\lambda|i)$.

\end{enumerate}
\end{prop}
\begin{prop}[Newton-Maclaurin inequality] For $\lambda\in \Gamma_{k}$ and $1\leq l \leq k\leq n-1$, we have
\[
\left(\frac{\sigma_{k}(\lambda)}{C^{k}_{n-1}}\right)^{\frac{1}{k}}\leq \left(\frac{\sigma_{l}(\lambda)}{C^{l}_{n-1}}\right)^{\frac{1}{l}}.
\]
\end{prop}

\begin{prop}[Concavity]
For any $k>l\geq 0$, suppose $A=A_{ij}\in \Theta_{n-1}$ such that $\lambda(A)\in \Gamma_{k}$. Then, we have
\[
\left[\frac{\sigma_{k}(A)}{\sigma_{l}(A)}\right]^{\frac{1}{k-l}}\ {\rm is \ a \ concave \ function \ in } \ \Gamma_{k}.
\]

\end{prop}

\section{The flow equation and related functional}
\label{Sec3}

To obtain the solvability of \eqref{WC}, we capitalize on the related flow equation.
Suppose $\Omega_{0}$  contains the origin in its interior, and is a smooth, strictly convex body in $\rnnn$. As presented in the introduction, we focus on a family of convex hypersurfaces $\partial \Omega_{t}$ parameterized by smooth map $X(\cdot,t):\sn\rightarrow \rnnn$ satisfying the following flow equation,
\begin{equation}\label{GCF2}
\left\{
\begin{array}{lr}
\frac{\partial X(x,t)}{\partial t}={\rm log}\left( \sigma_{k}(x,t)f(v) \right)v, \\
X(x,0)=\theta X_{0}(x).
\end{array}\right.
\end{equation}
Multiplying both sides of \eqref{GCF2} by $v$, by means of the definition of support function, the flow equation associated with the support function $h(x,t)$ of $\Omega_{t}$ is given by
\begin{equation}\label{GCF22}
\left\{
\begin{array}{lr}
\frac{\partial h(x,t)}{\partial t}={\rm log}(\sigma_{k}(x,t)f(x)), \\
h(x,0)=\theta h_{0}(x).
\end{array}\right.
\end{equation}
Now, establishing the functional associated with the flow \eqref{GCF2} as
\begin{equation}\label{func2}
J(t)=-W_{n-1-k}(\Omega_{t})+\int_{\sn}\frac{h(x,t)}{f(x)}dx,
\end{equation}
where $W_{n-1-k}(\Omega_{t})$ is the quermassintegral of $\Omega_{t}$, given by
\begin{equation}
\begin{split}
\label{wnk}
W_{n-1-k}(\Omega_{t})
&=\frac{1}{k+1}\int_{\sn}h(x,t)\sigma_{k}(x,t)dx.
\end{split}
\end{equation}

Under the flow, we obtain the monotonicity of the $J(t)$, which is showed as below.
\begin{lem}\label{monJ}
Assume that $\Omega_{t}$ contains the origin in its interior, and is a smooth, strictly convex solution satisfying the flow \eqref{GCF2} in $\rnnn$, and $J(t)$ is defined as \eqref{func2}. Then, $J(t)$ is nonincreasing  along the flow \eqref{GCF2}, i.e.,
 \[
 \frac{d}{dt}J(t)\leq 0.
 \]
\end{lem}
\begin{proof}
Using \eqref{GCF22} and \eqref{wnk}, we get
\begin{equation}
\begin{split}
\label{wkt}
\frac{d W_{n-1-k}(\Omega_{t})}{dt}&=\frac{1}{k+1}\frac{d}{dt}\int_{\sn}\sigma_{k}(x,t)h(x,t)dx\\
&=\frac{1}{k+1}\left( \int_{\sn}h\partial_{t}\sigma_{k} dx+\sigma_{k}\partial_{t}hdx \right)\\
&=\frac{1}{k+1}\left(\int_{\sn}\sum_{i,j}h\sigma^{ij}_{k}\nabla_{ij}(\partial_{t}h)dx+\int_{\sn}h\sum_{i,j}\sigma^{ij}_{k}\delta_{ij}(\partial_{t}h)dx\right.\\
&\quad  \left. +\int_{\sn}\sum_{i,j}\sigma^{ij}_{k}\partial_{t}h(w_{ij}-h\delta_{ij})dx-\int_{\sn}\sum_{i,j}\sigma^{ij}_{k}(\partial_{t}h)\nabla_{ij}hdx+\int_{\sn}\sigma_{k}\partial_{t}hdx \right)\\
&=\int_{\sn}\sigma_{k}\partial_{t}hdx,
\end{split}
\end{equation}
where in the last equality we employ the fact $\sum_{i}\nabla_{i}(\sigma^{ij}_{k})=0$. Therefore, taking the derivative of both side of \eqref{func2}  with respect to $t$, and utilizing \eqref{GCF22} and \eqref{wkt}, we have
\begin{equation}
\begin{split}
\label{}
\frac{d}{dt}J(t)&=-\int_{\sn}\left(\sigma_{k}(x,t)-\frac{1}{f(x)}\right)\frac{\partial h(x,t)}{\partial t}dx\\
&=-\int_{\sn}\frac{1}{f}\left(e^{h_{t}}-1\right)h_{t}dx\\
&\leq 0.
\end{split}
\end{equation}
Hence, the proof is completed.
\end{proof}
\section{Some apriori estimates}
\label{Sec4}
This section is delicated to obtaining some a priori estimates on the solution of \eqref{GCF2}.

 We begin with doing some preparations. In what follows, we always assume $h(x,t)$ is a smooth solution of $\eqref{GCF22}$ for $(x,t)\in \sn\times (0,\infty)$. Let $m=\inf f$ and $M=\sup f$ on $\sn$. It is clear to see that if the initial hypersurface $X_{0}$ is a sphere of radius $\rho_{0}> m^{-1/k}$, the solution $X(\cdot,t)$ to the equation
 \[
 \frac{\partial X}{\partial t}=\log (\sigma_{k} m)v, \quad X(\cdot,0)=X_{0},
 \]
remains to be spheres and the flow expands to infinity as $t\rightarrow \infty$. Similarly, if $X_{0}$ is a sphere of radius less than $M^{-1/k}$, the solution to
 \[
 \frac{\partial X}{\partial t}=\log (\sigma_{k} M)v, \quad X(\cdot,0)=X_{0},
 \]
is a family of spheres which shrinks to a point in finite time. Adopting the same idea exhibited in \cite{CW00}, we conclude that the solution $X(x,t)$ of $\eqref{GCF2}$ shall shrink to a point if $\theta$ is small enough, and shall expand to infinity if $\theta>0$ is large with the aid of comparison principle. Here we set
\[
\theta_{*}=\sup\{ \theta>0: X(\cdot,t) \  shrinks\ to \ a \ point \ in \ finite \ time\},
\]
and
\[
\theta^{*}=\inf\{ \theta>0: X(\cdot,t) \ expands \ to \  infinity  \ as \ t \rightarrow \infty \}.
\]

On the other hand, let $R(t)$ and $r(t)$ be the outer and inner radii of the hypersurface $\partial \Omega_{t}$ determined by $h(x,t)$ respectively. Set
\[
R_{0}=\sup \{R(t):t\in (0,\infty)\}
\]
and
\[
r_{0}=\inf \{r(t):t\in (0,\infty)\}.
\]
Lemma 2.2 of Chou-Wang \cite{CW00} have estimated the principal of curvature of $\Omega_{t}$ from both side in terms of $r^{-1}_{0}, R_{0}$ and initial data as follows.
\begin{lem}\label{chou-wang}
Let $r(t)$ and $R(t)$ be the inner and outer radii of a strictly convex hypersurface $\partial\Omega_{t}$ respectively. There exists a dimensional constant $C$ such that
\[
\frac{R^{2}(t)}{r(t)}\leq C\sup_{y\in \partial \Omega_{t}}\lambda_{y;\partial \Omega_{t}},
\]
where $\lambda_{y;\partial \Omega_{t}}$ is the maximal principal radius of $\partial \Omega_{t}$ at the point $y$.
\end{lem}

Here we give an alternate proof with respect to Lemma \ref{chou-wang} with beginning doing following preparation.

\begin{lem}\label{r1}
Consider the following rotationally symmetric ellipsoid in $\rnnn$:
\[
E_{B}=\{y\in \rnnn:|B^{-T}y|\leq 1\},
\]
where $B={\rm diag}(a,\cdots,a,b)$ is an $n\times n$ matrix with $a,b> 0$. For any point $y=(y^{'},y_{n})\in \partial E_{B}$, the maximal principal radius of curvature $\lambda_{y}$ has the estimate
\begin{equation}\label{B1}
\lambda_{y}\geq \frac{[b^{4}+(a^{2}-b^{2})y^{2}_{n}]^{\frac{3}{2}}}{ab^{4}}\geq \frac{|y^{'}|^{3}}{a^{3}}\times \frac{b^{2}}{a}.
\end{equation}
\end{lem}
\begin{proof}
Denote the support function of $E_{B}$ by $h_{B}$, there is
\[
h_{B}(x)=\sqrt{a^{2}|x^{'}|^{2}+b^{2}x^{2}_{n}},\quad x\in \sn.
\]
Let $x_{n}={\rm cos}\theta, \theta\in [0,\pi]$. Then, we obtain
\begin{equation}\label{B2}
h_{B}(\theta)=h_{B}(x^{'},x_{n})=\sqrt{a^{2}{\rm sin}^{2}\theta+b^{2}{\rm cos}^{2}\theta}, \quad \theta\in [0,\pi].
\end{equation}
Combining \eqref{B2}, the principal radius of curvature $h^{''}_{B}+h_{B}$ can be computed as
\begin{equation}\label{B3}
h^{''}_{B}+h_{B}=\frac{a^{2}b^{2}}{(a^{2}{\rm sin}^{2}\theta+b^{2}{\rm cos}^{2}\theta)^\frac{3}{2}}.
\end{equation}
For the normal $x$, the corresponding point $y\in \partial E_{B}$ is showed as
\begin{equation}\label{B4}
y=\frac{B^{T}Bx}{|Bx|}=\frac{(a^{2}x^{'},b^{2}x_{n})}{\sqrt{a^{2}|x^{'}|^{2}+b^{2}x^{2}_{n}}},
\end{equation}
\eqref{B4} reveals that
\begin{equation}\label{B5}
|y_{n}|=\frac{b^{2}|{\rm cos}\theta|}{\sqrt{a^{2}{\rm sin}^{2}\theta+b^{2}{\rm cos}^{2}\theta}}.
\end{equation}
Using \eqref{B3} and \eqref{B5}, we have
\begin{equation}\label{B6}
\lambda_{y}\geq a^{2}b^{2}\left( \frac{|y_{n}|}{b^{2}{|\rm cos}\theta|}\right)^{3}=\frac{a^{2}|y_{n}|^{3}}{b^{4}|{\rm cos \theta|^{3}}}.
\end{equation}
Applying $\eqref{B5}$ again, one has
\begin{equation}\label{B7}
|{\rm cos}\theta|=\frac{a|y_{n}|}{\sqrt{b^{4}+(a^{2}-b^{2})y^{2}_{n}}}.
\end{equation}
By virtue of \eqref{B6} and \eqref{B7}, it implies that the first inequality of \eqref{B1} holds. On the other hand, since
\begin{equation}\label{B8}
\frac{b^{4}+(a^{2}-b^{2})y^{2}_{n}}{b^{4}}=\frac{|y^{'}|^{2}}{a^{2}}+\frac{a^{2}-|y^{'}|^{2}}{b^{2}}\geq \frac{|y^{'}|^{2}}{a^{2}}.
\end{equation}
This gives the second inequality of \eqref{B1}.
\end{proof}
\begin{lem}\label{r2}
Let $\Omega_{1}, \Omega_{2}$ be two smooth convex bodies in $\rnnn$ satisfying $\Omega_{2}\subset \Omega_{1}$ and $\partial \Omega_{1}\cap \partial \Omega_{2}\neq \emptyset$. Then for any point $y\in \partial \Omega_{1}\cap \partial \Omega_{2}$, the maximal principal radius of curvature of $\partial \Omega_{1}$ at $y$ is larger than or equal to that of $\partial \Omega_{2}$ at $y$.
\end{lem}
\begin{proof}
This conclusion is obviously true for the plane case. The higher dimensional case is reduced to the plane case, and then follows.
\end{proof}
 Based on Lemmas \ref{r1} and \ref{r2}, we are in place to prove Lemma \ref{chou-wang}.

{ \bf \emph{ Proof of Lemma\ref{chou-wang}}.} Let $E$ be the maximum ellipsoid of $\Omega_{t}$. Without loss generality, we assume $E$ is given by
\[
E=\{y\in \rnnn: |B^{-T}y|\leq 1\},
\]
where $B={\rm diag}(b_{1},\ldots, b_{n-1},b_{n})$ is an $n\times n$ matrix with $0<b_{1}\leq \ldots \leq b_{n-1}\leq b_{n}$. Notice that $E\subset \Omega_{t} \subset nE$.

Let $E_{b}$ be the ellipsoid defined as
\[
E_{b}=\left\{ y\in \rnnn:\frac{|y^{'}|^{2}}{b^{2}}+\frac{y^{2}_{n}}{\left( \frac{b_{n}}{\sqrt{2}}\right)^{2}} \leq 1\right\}, \quad b\geq b_{1}.
\]
Then $E_{b_{1}}\subset E \subset \Omega_{t}$. Define
\[
\bar{b}=\max\{b\in \rnnn: E_{b}\subset \Omega_{t}\}.
\]
This  tells $E_{\bar{b}}\subset \Omega_{t} \subset nE$, which illustrates that $b_{1}\leq \bar{b}\leq nb_{1}$. Moreover, $\partial E_{\bar{b}}\cap \partial \Omega_{t}\neq \emptyset$, and let $\bar{y}\in \partial E_{\bar{b}}\cap \partial \Omega_{t}$. Using Lemma \ref{r2}, there is
\begin{equation}\label{t1}
\lambda_{\bar{y};\partial \Omega_{t}}\geq \lambda_{\bar{y};\partial E_{\bar{b}}}.
\end{equation}

We now estimate $\lambda_{\bar{y};\partial E_{\bar{b}}}$. Write $\bar{y}=(\bar{y}^{'},\bar{y}_{n})$. Recalling $\bar{y}\in \partial \Omega_{t}$ and $\Omega_{t} \supset E$, $\bar{y}$ is not an interior point of $E$, this implies $|B^{-T}\bar{y}|\geq 1$, i.e.,
\begin{equation}\label{t2}
\frac{\bar{y}_{1}^{2}}{b^{2}_{1}}+\cdots +\frac{\bar{y}_{n-1}^{2}}{b^{2}_{n-1}}+\frac{\bar{y}_{n}^{2}}{b^{2}_{n}}\geq 1.
\end{equation}
Due to $b_{n-1}\geq \cdots \geq b_{1}\geq \frac{\bar{b}}{n}$, \eqref{t2} says that
\begin{equation}\label{t3}
\frac{|\bar{y}^{'}|^{2}}{\left(\frac{\bar{b}}{n}\right)^{2}}+\frac{\bar{y}^{2}_{n}}{b^{2}_{n}}\geq 1.
\end{equation}
From $y\in \partial E_{\bar{b}}$, there is
\begin{equation}\label{t4}
\frac{|\bar{y}^{'}|^{2}}{\bar{b}^{2}}+\frac{\bar{y}^{2}_{n}}{\left(\frac{b_{n}}{\sqrt{2}}\right)^{2}}=1.
\end{equation}
Inserting \eqref{t4} into \eqref{t3}, we obtain
\begin{equation}\label{t5}
\frac{|\bar{y}^{'}|^{2}}{\bar{b}}\geq \frac{1}{\sqrt{2n^{2}-1}}> \frac{1}{\sqrt{2n}}.
\end{equation}
Combining \eqref{t5} and Lemma \ref{r1}, there is
\[
\lambda_{\bar{y};\partial E_{\bar{b}}}\geq \frac{|\bar{y}^{'}|^{3}}{\bar{b}^{3}}\times \frac{b^{2}_{n}}{2\bar{b}}> \frac{b^{2}_{n}}{4\sqrt{2}n^{3}\bar{b}}\geq \frac{b^{2}_{n}}{4\sqrt{2}n^{4}b_{1}}.
\]
The above inequality together with \eqref{t1}, which gives
\[
\lambda_{y;\partial \Omega_{t}}> \frac{b^{2}_{n}}{4\sqrt{2}n^{4}b_{1}}.
\]
Note that $b_{1}\leq r \leq R\leq nb_{n}$. Therefore,
\[
\lambda_{y;\partial \Omega_{t}}> \frac{R^{2}}{4\sqrt{2}n^{6}r}.
\]
Hence, the proof of Lemma \ref{chou-wang} is finished.

 In what follows, we shall give an upper estimate for the principal radii of curvature.
\begin{lem}\label{PCUL**}
  Assume that $\Omega_{t}$ contains the origin in its interior, and is a smooth, strictly convex solution satisfying the flow \eqref{GCF2} in $\rnnn$. Let $f(x)$ be a smooth and positive function on the unit sphere ${\sn}$. Then there exists a positive constant $C$, independent of $t$, such that
\begin{equation}\label{PUL}
\lambda (\{w_{ij}\})\leq C(1+R_{0})^{\frac{3}{2}},
\end{equation}
where $w_{ij}=h_{ij}+h\delta_{ij}$, and $\lambda (\{w_{ij}\})$ are the eigenvalues of $\{w_{ij}\}$.
\end{lem}

\begin{proof}
To obtain \eqref{PUL}, we are in the position to set the auxiliary function as
\begin{equation}\label{LFun}
E(x,t)=\lambda_{max}(\{w_{ij}\})+|\nabla_{\sn} h|^{2}+h^{2},
\end{equation}
where $\lambda_{max}(\{w_{ij}\})$ is the maximal eigenvalue of $\{w_{ij}\}$.

Now we shall assume that the maximum of $E(x,t)$ is achieved at $(x_{0},t_{0})$ on $\sn\times (0,\infty)$. By virtue of a rotation of coordinates, we may suppose that $\{w_{ij}(x_{0},t_{0})\}$ is diagonal, and $\lambda_{max}(\{w_{ij}\})(x_{0},t_{0})=w_{11}(x_{0},t_{0})$. So, $\eqref{LFun}$ becomes the following form:
\begin{equation}\label{LFun2}
\widetilde{E}(x,t)=w_{11}+|\nabla_{\sn} h|^{2}+h^{2}.
\end{equation}
Since  $\widetilde{E}(x,t)$ attains the maximum at $(x_{0},t_{0})$, it follows that, at $(x_{0},t_{0})$, there is
\begin{equation}
\begin{split}
\label{gaslw1}
0=\nabla_{i}\widetilde{E}&=\nabla_{i}w_{11}+2\sum_{j} h_{j}h_{ji}+2hh_{i}\\
&=(h_{i11}+h_{1}\delta_{1i})+2h_{i}h_{ii}+2hh_{i},
\end{split}
\end{equation}
and we also have
\begin{align}\label{gaslw2}
0&\geq \nabla_{ii}\widetilde{E}=\nabla_{ii}w_{11}
+2\left(\sum_{j} h_{j}h_{jii}+h^{2}_{ii}\right)+2h^{2}_{i}+2hh_{ii}.
\end{align}
Furthermore,
\begin{equation}
\label{gaslw3}
0\leq\partial_{t}\widetilde{E}=\partial_{t}w_{11}+2\sum_{j} h_{j}h_{jt}+2hh_{t}=(h_{11t}+h_{t})+2\sum_{j} h_{j}h_{jt}+2hh_{t}.
\end{equation}
Due to
\begin{equation}\label{gaslw4}
h_{t}=\log \sigma_{k}(x,t)+\log f(x).
\end{equation}
Taking the covariant derivative of \eqref{gaslw4} with respect to $e_{j}$, yields
\begin{align}\label{gaslw6}
h_{tj}&=\sum_{i,k} \frac{\sigma_{k}^{ik}}{\sigma_{k}}\nabla_{j}w_{ik}+\nabla_{j}\log f\notag\\
&=\sum_{i} \frac{\sigma_{k}^{ii}}{\sigma_{k}}(h_{jii}+h_{i}\delta_{ij})+\nabla_{j}\log f,
\end{align}
where $\sigma^{ik}_{k}=\frac{\partial \sigma_{k}}{\partial w_{ik}}$, and
\begin{equation}
\begin{split}
\label{gaslw7}
&h_{11t}=\sum_{i,k}\frac{\sigma_{k}^{ik}}{\sigma_{k}}\nabla_{11}w_{ik}+\sum_{i,k,\beta,s}\frac{\sigma_{k}^{ik,\beta s}}{\sigma_{k}}\nabla_{1}w_{ik}\nabla_{1}w_{\beta s}-\sum_{i,k,\beta,s}\frac{\sigma_{k}^{ik}\sigma_{k}^{\beta s}}{\sigma_{k}^{2}}\nabla_{1}w_{ik}\nabla_{1}w_{\beta s}+\nabla_{11}\log f\\
&=\sum_{i}\frac{\sigma_{k}^{ii}}{\sigma_{k}}\nabla_{11}w_{ii}+\sum_{i,k,\beta,s}\frac{\sigma_{k}^{ik,\beta s}}{\sigma_{k}}\nabla_{1}w_{ik}\nabla_{1}w_{\beta s}-\sum_{i,s}\frac{\sigma_{k}^{ii}\sigma_{k}^{s s}}{\sigma_{k}^{2}}\nabla_{1}w_{ii}\nabla_{1}w_{s s}+\nabla_{11}\log f,
\end{split}
\end{equation}
where $\sigma_{k}^{ik,\beta s}=\frac{\partial^{2}\sigma_{k}}{\partial w_{ik}\partial w_{\beta s}}$.  On the other hand, the Ricci identity on sphere reads
\begin{equation*}
\nabla_{11}w_{ij}=\nabla_{ij}w_{11}-\delta_{ij}w_{11}+\delta_{11}w_{ij}-\delta_{1i}w_{1j}+\delta_{1j}w_{1i}.
\end{equation*}
So, with the aid of the Ricci identity, \eqref{gaslw2}, \eqref{gaslw3}, \eqref{gaslw6}, \eqref{gaslw7} and the concavity of $\log\sigma_{k}(x,t)$, at $(x_{0},t_{0})$, we obtain
\begin{equation}
\begin{split}
\label{gaslw8}
0&\geq \sum_{i}\frac{\sigma_{k}^{ii}}{\sigma_{k}}\nabla_{ii}\widetilde{E}-\partial_{t}\widetilde{E}\\
&=\sum_{i}\frac{\sigma_{k}^{ii}}{\sigma_{k}}\nabla_{ii}w_{11}+2\sum_{i}\frac{\sigma_{k}^{ii}}{\sigma_{k}}(\sum_{j}h_{j}h_{jii}+h^{2}_{ii})+2\sum_{i}\frac{\sigma_{k}^{ii}}{\sigma_{k}}h^{2}_{i}+2\sum_{i}\frac{\sigma_{k}^{ii}}{\sigma_{k}}hh_{ii}\\
&\quad -w_{11t}-2\sum_{j}h_{j}h_{jt}-2hh_{t}\\
&=\sum_{i}\frac{\sigma_{k}^{ii}}{\sigma_{k}}(\nabla_{11}w_{ii}+w_{11}-w_{ii})+2\sum_{j}h_{j}(\sum_{i}\frac{\sigma_{k}^{ii}}{\sigma_{k}}h_{jii}-h_{jt})+2\sum_{i}\frac{\sigma_{k}^{ii}}{\sigma_{k}}h^{2}_{ii}\\
&\quad +2\sum_{i}\frac{\sigma_{k}^{ii}}{\sigma_{k}}h^{2}_{i}+2\sum_{i}\frac{\sigma_{k}^{ii}}{\sigma_{k}}h(w_{ii}-h)-h_{11t}-h_{t}-2hh_{t}\\
&\geq\sum_{i}\frac{\sigma_{k}^{ii}}{\sigma_{k}}\nabla_{11}w_{ii}-h_{11t}+2\sum_{j}h_{j}(-\frac{\sigma_{k}^{jj}}{\sigma_{k}}h_{j}-\nabla_{j}\log f)+2\sum_{i}\frac{\sigma_{k}^{ii}}{\sigma_{k}}(w_{ii}-h)^{2}\\
&\quad +2\sum_{i}\frac{\sigma_{k}^{ii}}{\sigma_{k}}h^{2}_{i}+2\sum_{i}\frac{\sigma_{k}^{ii}w_{ii}}{\sigma_{k}}h-2\sum_{i}\frac{\sigma_{k}^{ii}}{\sigma_{k}}h^{2}-(2h+1)h_{t}\\
&=-\sum_{i,k,\beta,s}\frac{\sigma_{k}^{ik,\beta s}}{\sigma_{k}}\nabla_{1}w_{ik}\nabla_{1}w_{\beta s}+\sum_{i,s}\frac{\sigma_{k}^{ii}\sigma_{k}^{s s}}{\sigma_{k}^{2}}\nabla_{1}w_{ii}\nabla_{1}w_{s s}-\nabla_{11}\log f-2\sum_{j}h_{j}\nabla_{j}\log f\\
&\quad +2\sum_{i}\frac{\sigma_{k}^{ii}w^{2}_{ii}}{\sigma_{k}}-2\sum_{i}\frac{\sigma_{k}^{ii}w_{ii}}{\sigma_{k}}h-(2h+1)h_{t}\\
&\geq 2\sum_{i}\frac{\sigma_{k}^{ii}w^{2}_{ii}}{\sigma_{k}}-\nabla_{11}\log f-2\sum_{j}h_{j}\nabla_{j}\log f-2\sum_{i}\frac{\sigma_{k}^{ii}w_{ii}}{\sigma_{k}}h-(2h+1)h_{t}.
\end{split}
\end{equation}
This implies that
\begin{equation}\label{Fmax}
2\frac{\sigma_{k}^{11}w_{11}^{2}}{\sigma_{k}}\leq 2\sum_{j} h_{j}\nabla_{j}\log f +\nabla _{11}\log f+2\sum_{i}\frac{\sigma_{k}^{ii}w_{ii}}{\sigma_{k}}h+(2h+1)h_{t}.
\end{equation}
 Since $\{w_{ij}\}$ is diagonal at $(x_{0},t_{0})$, and $w_{ii}=\lambda_{i}$, $\forall i=1,\ldots, n-1$. Then we obtain
\begin{equation}\label{F11}
\sigma_{k}^{11}=\sigma_{k-1}(\lambda|1),
\end{equation}
where $\lambda:=(\lambda_{1},\ldots,\lambda_{n-1})$, $\sigma_{k-1}(\lambda|1)$ denotes the $(k-1)$-th symmetric functions with $\lambda_{1}=0$. By using Newton-MacLaurin inequality\cite{CW01}, one has
\begin{equation}\label{Mac}
\left[\frac{\sigma_{k-1}(\lambda|1)}{C^{k-1}_{n-1}}\right]^\frac{1}{k-1}\geq \left[\frac{\sigma_{k}(\lambda|1)}{C^{k}_{n-1}}\right]^\frac{1}{k},
\end{equation}
using \eqref{Mac}, for some dimensional positive constants $C$, we get
\begin{equation}
\begin{split}
\label{Fw11}
\sigma_{k}(\lambda|1)\leq C \sigma_{k-1}(\lambda|1)^\frac{k}{k-1}\leq C\lambda_{1}\sigma_{k-1}(\lambda|1).
\end{split}
\end{equation}
Substituting \eqref{Fw11} into $\sigma_{k}(\lambda)=\sigma_{k}(\lambda|1)+\lambda_{1}\sigma_{k-1}(\lambda|1)$, we have
\begin{equation}\label{1cq}
\lambda_{1}\sigma_{k-1}(\lambda|1)\geq C \sigma_{k}(\lambda).
\end{equation}
So, employing \eqref{1cq}, one see that
\begin{equation}
\begin{split}
\label{maxw11}
\frac{\sigma_{k}^{11}w^{2}_{11}}{\sigma_{k}}&=\frac{\sigma_{k-1}(\lambda|1)\lambda^{2}_{1}}{\sigma_{k}}\geq \frac{C\sigma_{k}\lambda_{1}}{\sigma_{k}}=Cw_{11}.
\end{split}
\end{equation}
Hence, substituting \eqref{maxw11} into \eqref{Fmax}, there is
\begin{equation}\label{CW11}
2Cw_{11}\leq 2\sum h_{j}\nabla_{j}\log f +\nabla _{11}\log f+2kh+(2h+1)h_{t}.
\end{equation}
Due to
\begin{equation*}
\nabla_{j}\log f=\frac{f_{j}}{f},
\end{equation*}
and
\begin{equation*}
\begin{split}
\nabla_{11}\log f&=\frac{ff_{11}-f^{2}_{1}}{f^{2}}.
\end{split}
\end{equation*}
Hence
\begin{equation}
\begin{split}
\label{gaslw9}
&2\sum_{j} h_{j}\nabla_{j}\log f+\nabla_{11}\log f\\
&=2\sum_{j} h_{j}\frac{f_{j}}{f}+\frac{ff_{11}-f^{2}_{1}}{f^{2}}.
\end{split}
\end{equation}
Recall \eqref{GCF22}, at $(x_{0},t_{0})$, it suffices to have
\begin{equation}\label{ht}
 h_{t}\leq C_{0}+C_{1}{\rm log} w_{11}.
 \end{equation}
 Now, substituting \eqref{gaslw9} and \eqref{ht} into \eqref{CW11}, we have
\begin{equation}
\begin{split}
\label{gaslw10}
w_{11}&\leq C(2R_{0}+1)(C_{0}+C_{1}{\rm log}w_{11})+C_{3}R_{0}+C_{4}\\
& \leq \widetilde{C}_{0}+\widetilde{C}_{1}R_{0}+\widetilde{C}_{3}R_{0}{\rm log}w_{11}+\widetilde{C}_{4}{\rm log}w_{11},
\end{split}
\end{equation}
 provided $w_{11}\gg 1$, at $(x_{0},t_{0})$, there is
\begin{equation}
\begin{split}
\label{es}
w^{\frac{2}{3}}_{11}\leq \bar{C}_{0}+\bar{C}_{1}R_{0}.
\end{split}
\end{equation}
From \eqref{es}, it yields
\begin{equation*}
w_{11}\leq C(1+R_{0})^{\frac{3}{2}}
\end{equation*}
for a positive constant $C$, depending only on $||\log f||_{C^2}$, $n$, $k$. Therefore, the proof is completed.
\end{proof}
In conjunction with Lemma \ref{chou-wang} and Lemma \ref{PCUL**}, we get the following non-collapsing estimate.

\begin{lem}\label{LoUp}
 There exists a positive constant $C$, independent of $t$, such that
\[
r(t)\geq \frac{R^{2}(t)}{C(1+R_{0})^\frac{3}{2}}.
\]
\end{lem}
In the spirit of  Lemma \ref{LoUp}, we assert that for any $\theta\in [\theta_{*}, \theta^{*}]$, the inner radii of $\partial \Omega_{t}$ have a uniform positive lower bound and the outer radii are uniformly bound from above. Then the $C^{1}$ of $h$ immediately follows with the aid of the convexity of $h$, i,e., $|\nabla_{\sn}h|\leq \max_{\sn} h$.

Next, we are ready to obtain the uniform lower bound of $\sigma_{k}(x,t)$.

\begin{lem}\label{prin}
 Let $\Omega_{t}$ contain the origin in its interior, and be a smooth, strictly convex solution satisfying the flow \eqref{GCF2} in $\rnnn$. Let $f(x)$ be a smooth and positive function on the unit sphere ${\sn}$. Then there exists a positive constant $\bar{C}$, independent of $t$, such that
\[
\sigma_{k}\geq \bar{C}.
\]

\end{lem}
\begin{proof}
Set
\begin{equation}\label{}
\zeta(t)=\frac{1}{\omega_{n}}\int_{\sn}xh(x,t)dx
\end{equation}
be the Steiner point of $\Omega_{t}$. By means of Lemma \ref{LoUp},  there exists a positive $\varepsilon_{0}$, which depends only on $n$, $r_{0}$ and $R_{0}$, independent of $t$,  such that
\[
h(x,t)-\zeta(t)\cdot x\geq 2\varepsilon_{0}.
\]
Now establishing the auxiliary function as
\begin{equation}\label{AF1}
Q(x,t)=\frac{-{\rm log}\sigma_{k}-{\rm log} f}{h-\zeta(t)\cdot x-\varepsilon_{0}}=\frac{-h_{t}}{h-\zeta(t)\cdot x-\varepsilon_{0}}.
\end{equation}
For any fixed $t\in(0,+\infty)$, suppose that the (positive) maximum of $Q(x,t)$ is achieved at $x_{0}$. Rotate the axes so that $\{w_{ij}\}$ is diagonal at $x_{0}$. Thus, we get that at $x_{0}$,
\begin{equation}\label{Up1}
0=\nabla_{i}Q=\frac{-h_{ti}}{h-\zeta_{0}-\varepsilon_{0}}+\frac{h_{t}(h_{i}-\zeta_{i})}{(h-\zeta_{0}-\varepsilon_{0})^{2}},
\end{equation}
where $\zeta_{i}:=\zeta\cdot e_{i}$, and $\zeta_{0}=\zeta\cdot x_{0}$.

Then, applying \eqref{Up1}, at $x_{0}$, we also have
\begin{equation}
\begin{split}
\label{Up2}
0\geq\nabla_{ij}Q&=\frac{-h_{tij}}{h-\zeta_{0}-\varepsilon_{0}}+\frac{h_{ti}(h_{j}-\zeta_{j})+h_{tj}(h_{i}-\zeta_{i})+h_{t}(h_{ij}+\zeta_{0}\delta_{ij})}{(h-\zeta_{0}-\varepsilon_{0})^{2}}-\frac{2h_{t}(h_{i}-\zeta_{i})(h_{j}-\zeta_{j})}{(h-\zeta_{0}-\varepsilon_{0})^{3}}\\
&=\frac{-h_{tij}}{h-\zeta_{0}-\varepsilon_{0}}+\frac{h_{t}h_{ij}}{(h-\zeta_{0}-\varepsilon_{0})^{2}}+\frac{h_{t}\zeta_{0}\delta_{ij}}{(h-\zeta_{0}-\varepsilon_{0})^{2}}.
\end{split}
\end{equation}
 \eqref{Up2} implies
\begin{equation}
\begin{split}
\label{Up3}
-h_{tij}-h_{t}\delta_{ij}&\leq - \frac{h_{t}h_{ij}}{h-\zeta_{0}-\varepsilon_{0}}-\frac{h_{t}\zeta_{0}\delta_{ij}}{(h-\zeta_{0}-\varepsilon_{0})}-h_{t}\delta_{ij}\\
&=\frac{-h_{t}}{h-\zeta_{0}-\varepsilon_{0}}[h_{ij}+(h-\zeta_{0}-\varepsilon_{0})\delta_{ij})]+Q\zeta_{0}\delta_{ij}\\
&=Q(w_{ij}-\varepsilon_{0}\delta_{ij}-\zeta_{0}\delta_{ij})+Q\zeta_{0}\delta_{ij}\\
&=Q(w_{ij}-\varepsilon_{0}\delta_{ij}).
\end{split}
\end{equation}
On the other hand,  at $x_{0}$, we have
\begin{equation}
\begin{split}
\label{s2k} \partial_{t}Q&=\frac{-h_{tt}}{h-\zeta_{0}-\varepsilon_{0}}+\frac{h^{2}_{t}}{(h-\zeta_{0}-\varepsilon_{0})^{2}}-\frac{h_{t}\frac{d\zeta_{0}}{dt}}{(h-\zeta_{0}-\varepsilon_{0})^{2}}\\
&=\frac{1}{h-\zeta_{0}-\varepsilon_{0}}\left(\sigma_{k}\frac{ \partial \sigma^{-1}_{k}}{\partial t}\right)+Q^{2}-\frac{h_{t}\frac{d\zeta_{0}}{dt}}{(h-\zeta_{0}-\varepsilon_{0})^{2}}.
\end{split}
\end{equation}
Then, we derive
\begin{equation}
\begin{split}
\label{sigmak41}
\frac{\partial \sigma_{k}^{-1}}{\partial t}&=-\sigma^{-2}_{k}\sum_{i,j}\frac{\partial \sigma_{k}}{\partial w_{ij}}(h_{tij}+h_{t}\delta_{ij})\\
&\leq \sigma^{-2}_{k}\sum_{i,j}\frac{\partial \sigma_{k}}{\partial w_{ij}}(w_{ij}-\varepsilon_{0}\delta_{ij})Q\\
&=\sigma^{-2}_{k}\left(k\sigma_{k}-\varepsilon_{0}\sum_{i} \sigma^{ii}_{k}\right)Q\\
&\leq (k\sigma^{-1}_{k}-C\varepsilon_{0}\sigma^{-1-\frac{1}{k}}_{k})Q.
\end{split}
\end{equation}
The last inequality we use Newton-MacLaurin inequality $\sum_{i}\sigma^{ii}_{k}\geq C\sigma^{1-\frac{1}{k}}_{k}$, where $C$ is a positive constant depending only on $n$ and $k$. Applying \eqref{sigmak41} into \eqref{s2k}, we get
\begin{equation}
\begin{split}
\label{ODE}
\partial_{t}Q&\leq \frac{k}{h-\zeta_{0}-\varepsilon_{0}}Q-\frac{C\varepsilon_{0}}{h-\zeta_{0}-\varepsilon_{0}}\sigma_{k}^{-\frac{1}{k}}Q+Q^{2}+\frac {|h_{t}||\frac{d \zeta_{0}}{dt}|}{(h-\zeta_{0}-\varepsilon_{0})^{2}}\\
&\leq C_{1}Q-C_{2}Qe^{\frac{\varepsilon_{0}Q}{k}}+C_{3}Q^{2}<0
\end{split}
\end{equation}
providing $Q\gg 1$ for some positive $C_{1},C_{2},C_{3}$. Therefore, \eqref{ODE} implies that
\[
Q(x_{0},t)\leq C
\]
for $C>0$. This implies that
\[
\sigma_{k}\geq \bar{C}
\]
for $\bar{C}>0$, independent of $t$, depending only on $||\log f||_{C^{2}}$, $n$, $k$. The proof is completed.
\end{proof}

Combining Lemmas \ref{PCUL**}, \ref{LoUp} and \ref{prin}, it suffices to obtain the uniform lower and upper bounds of $\sigma_{k}(x,t)$, as
\begin{equation}\label{KUL}
\tilde{C}_{1}\leq \sigma_{k}\leq \tilde{C}_{2}
\end{equation}
for positive $\tilde{C}_{1}$ and $\tilde{C}_{2}$, depending only on $||\log f||_{C^{2}}$, $n$, $k$.

\section*{Acknowledgment}The authors would like to thank their supervisor Prof. Yong Huang for valuable comments regarding the exposition of this paper. Also grateful for Prof. Jian Lu for providing illuminating thoughts on the proof of Lemma \ref{chou-wang}.

\end{document}